\documentclass[11pt]{amsart} 
\usepackage{amsmath,amsthm}
\usepackage{xypic}
\usepackage{enumerate}

\usepackage{amscd}
\usepackage{amsfonts,amssymb,latexsym} % \Box, \twoheadrightarrow
\usepackage[all]{xy} 
\xyoption{arc}
\CompileMatrices

\newcommand{\udots}{\mathinner{\mskip1mu\raise1pt\vbox{\kern7pt\hbox{.}}
\mskip2mu\raise4pt\hbox{.}\mskip2mu\raise7pt\hbox{.}\mskip1mu}}

\newcommand{\SA}{{\mathcal{A}}}

\newcommand{\SD}{{\mathcal{D}}}
\newcommand{\SE}{{\mathcal{E}}}

\newcommand{\SH}{{\mathcal{H}}}

\newcommand{\SM}{{\mathcal{M}}}

\newcommand{\SO}{{\mathcal{O}}}
\newcommand{\SP}{{\mathcal{P}}}

\newcommand{\ST}{{\mathcal{T}}}
\newcommand{\SU}{{\mathcal{U}}}
\newcommand{\SV}{{\mathcal{V}}}
\newcommand{\SW}{{\mathcal{W}}}

\newcommand{\SZ}{{\mathcal{Z}}}

\newcommand{\ZZ}{\mathbb{Z}}

\newcommand{\codim}{\operatorname{codim}}

\newcommand{\Pic}{\operatorname{Pic}}

\newcommand{\id}{\operatorname{Id}}

\newcommand{\Aut}{\operatorname{Aut}}

\newcommand{\rk}{\operatorname{rk}}
\newcommand{\pdeg}{\operatorname{pardeg}}

\newcommand{\wt}{\widetilde}

\newcommand{\op}{\operatorname}

\newtheorem{proposition}{Proposition}[section]
\newtheorem{theorem}[proposition]{Theorem}
\newtheorem{definition}[proposition]{Definition}

\newtheorem{lemma}[proposition]{Lemma}

\newtheorem{corollary}[proposition]{Corollary}
\newtheorem{remark}[proposition]{Remark}

\numberwithin{equation}{section}

\usepackage[backend=bibtex,style=alphabetic,citestyle=alphabetic]{biblatex}

\usepackage{hyperref}
\hypersetup{
    colorlinks=false,
    pdfborder={0 0 0}
}

\title[Automorphisms moduli of parabolic bundles with fixed degree]{Automorphism group of the moduli space of parabolic vector bundles with fixed degree}
\author[D. Alfaya]{David Alfaya}

\date{}

\address{D. Alfaya, 
\newline\indent
Department of Applied Mathematics and Institute for Research in Technology, ICAI School of Engineering, Comillas Pontifical University, C/Alberto Aguilera 25, 28015 Madrid, Spain}

\email{dalfaya@comillas.edu}

\begin{document}

\begin{abstract}
We find all possible isomorphisms and $3$-birational maps (i.e., birational maps which induce an isomorphism between open subsets whose respective complements have codimension at least $3$) between moduli spaces of parabolic vector bundles with fixed degree. We prove that every $3$-birational map can be described as a composition of tensorization by a fixed line bundle, Hecke transformations, dualization, taking pullback by an isomorphism between the curves and the action of the group of automorphisms of the Jacobian variety of the curve which fix the $r$-torsion. In particular, we prove a Torelli type theorem, stating that the $3$-birational class of the moduli space determines the isomorphism class of the curve.
\end{abstract}

\maketitle

\section{Introduction}

Let $X$ be a smooth complex projective algebraic curve. The classical Torelli theorem states the isomorphism class of $J(X)$ as a polarized variety identifies unequivocally the isomorphism class of the curve $X$. There are several generalizations of this result to higher rank \cite{MN68, Tyu70, NR75, KP95, HR04, BGM13}, proving that the moduli space $\SM(X,r,\xi)$ of vector bundles of rank $r$ with fixed determinant $\xi$,  or the moduli space $\SM(X,r,d)$ of vector bundles of rank $r$ with fixed degree $d$ identify the isomorphism class of $X$. Moreover, in \cite{KP95} the automorphisms of $\SM(X,r,\xi)$ and $\SM(X,r,d)$ were computed and, later on, \cite{HR04} and \cite{BGM13} gave different proofs of these results in the fixed determinant case.

Recently, some of these results have been extended to moduli spaces of parabolic bundles with fixed determinant. Let us assume that the curve $X$ has a finite set of marked points $D\subset X$, and let us consider the moduli space $\SM(X,D,r,\alpha,\xi)$ of parabolic vector bundles on $(X,D)$ of rank $r$, system of weights $\alpha$ and fixed determinant $\xi$. In \cite{BBB01}, Balaji, del Baño and Biswas proved a Torelli type result for this moduli space under the conditions that $r=2$, $\deg(\xi)$ is even and $\alpha$ is small, stating that the isomorphism class of the pair $(X,D)$ can be recovered from the isomorphism class of $\SM(X,D,2,\alpha,\xi)$. This result was generalized in \cite{AG19} to arbitrary rank, degree and generic parabolic weights. \cite{AG19} also provided an explicit description of all possible isomorphisms $\SM(X,D,r,\alpha,\xi)\stackrel{\sim}{\longrightarrow} \SM(X',D',r',\alpha',\xi')$ between different moduli spaces of parabolic vector bundles with fixed determinant built from possibly different parameters. In particular, the automorphism group of the moduli space was computed, generalizing \cite{KP95} to moduli spaces of parabolic bundles with fixed determinant. More precisely, it was proven that all possible isomorphisms are given by suitable compositions of the following basic transformations
\begin{itemize}
\item Taking the pullback with respect to an isomorphism $\sigma:X'\longrightarrow X$, such that $\sigma(D')=D$,
\item tensoring with a line bundle,
\item taking the dual parabolic bundle, and 
\item applying a Hecke transformation to the parabolic vector bundle at a parabolic point.
\end{itemize}

Moreover, \cite{AG19} also provided a $3$-birational version of these results, classifying all possible $3$-birational maps $\SM(X,D,r,\alpha,\xi)\dashrightarrow \SM(X',D',r',\alpha',\xi')$. By $3$-birational map we mean the following. Given two algebraic varieties $\SM,\SM'$ and an integer $k>0$, a $k$-birational map between $\SM$ and $\SM'$ is defined as an isomorphism $\Phi:\SU\longrightarrow \SU'$ between open subsets $\SU\subset \SM$ and $\SU'\subset \SM'$ such that their respective complements have codimension at least $k$, i.e., such that
$$\codim(\SM\backslash \SU , \SM)\ge k, \quad \quad \codim(\SM'\backslash \SU',\SM')\ge k.$$
We say that $\SM$ and $\SM'$ are $k$-birational if there exists a $k$-birational map between them. In particular, if $\SM$ and $\SM'$ are connected, then they are birational if and only if they are $1$-birational.

The results from \cite{BBB01} and \cite{AG19} are restricted to moduli spaces of parabolic bundles with fixed determinant. The objective of this work is to extend these results to the moduli space $\SM(X,D,r,\alpha,d)$ of $\alpha$-stable parabolic vector bundles with fixed degree $d$ and rank $r$, classifying all possible isomorphisms between different moduli spaces of this type.

In this case, there exists an additional family of transformations which induce automorphisms of the moduli space. Given an automorphism of the Jacobian of $X$, $\rho\in \Aut(J(X))$ which acts as the identity on the $r$-torsion points $J(X)[r]$, and a fixed line bundle $\xi\in \Pic^d(X)$, we can build an automorphism $\SA^\xi_\rho:\SM(X,D,r,\alpha,d)\stackrel{\sim}{\longrightarrow} \SM(X,D,r,\alpha,d)$ which satisfies the following. If $(E,E_\bullet)\in \SM(X,D,r,\alpha,\xi)$ and $L\in J(X)$, then
$$\SA^\xi_\rho((E,E_\bullet)\otimes L)=(E,E_\bullet)\otimes \rho(L)$$

In \cite{AG19}, it is observed that in the study of maps between moduli spaces of parabolic vector bundles it is more natural to work directly in the framework of $k$-birational geometry, as it allows the usage of more flexible geometric strategies and, moreover, it simplifies significantly the additional problems arising from the existence of stability parameters. This paradigm has also been successfully applied in the analysis of other moduli spaces of decorated bundles, like framed bundles \cite{AB21}. In this paper we will follow the same approach, building the necessary technical lemmata from a $k$-birational perspective and then extending the results to globally defined isomorphisms, thus obtaining at the same time results in the categories of $3$-birational maps and algebraic isomorphisms. This turns out to be crucial for the proof. Even if we only worked with isomorphisms, some important lemmata (like Lemma \ref{lemma:JacobianOrbit} and Lemma \ref{lemma:JacobianOrbitsFixed}) would still rely on a birational approach to succeed.

More precisely, we obtain the following $3$-birational Torelli theorem (Theorem \ref{thm:TorelliBir}).

\begin{theorem}
Let $(X,D)$ and $(X',D')$ be two marked smooth complex projective curves of genus at least $4$ endowed with full flag systems of weights $\alpha$ and $\alpha'$ respectively. Then $\SM(X,D,r,\alpha,d)$ and $\SM(X',D',r',\alpha',d')$ are $3$-birational if and only if
\begin{enumerate}
\item $r=r'$ and ,
\item There exits an isomorphism $\sigma:X\longrightarrow X'$ such that $\sigma(D)=D'$
\end{enumerate}
\end{theorem}

observe that, in this case, the Torelli Theorem admits a reciprocal, i.e., there is a bijective correspondence between $3$-birational classes of moduli spaces of parabolic bundles with fixed degree and pairs consisting on an isomorphism class of the a marked curve $(X,D)$ and a rank $r$. As a particular case, we can deduce the usual Torelli theorem (Corollary \ref{cor:Torelli}).

\begin{theorem}
Let $(X,D)$ and $(X',D')$ be two marked smooth complex projective curves of genus at least $4$ endowed with full flag systems of weights $\alpha$ and $\alpha'$ respectively. If $\SM(X,D,r,\alpha,d) \cong \SM(X',D',r',\alpha',d')$ then $r=r'$ and $(X,D)\cong (X',D')$.
\end{theorem}

Moreover, as the main result of this paper, we refine this result to obtain a full explicit description of all possible $3$-birational maps and isomorphisms between moduli spaces of parabolic bundles with fixed degree built with possibly different invariants (Theorem \ref{thm:3birMaps} and Corollary \ref{cor:maps}).
\begin{theorem}
If $(X,D)$ and $(X',D')$ are smooth complex projective curves of genus at least $6$ with full flag generic systems of weights $\alpha$ and $\alpha'$. If $\Phi:\SM(X,D,r,\alpha,d) \dashrightarrow \SM(X',D',r',\alpha',d')$ is a $3$-birational map then $r=r'$ and there exist
\begin{itemize}
\item an isomorphism $\sigma:(X,D)\cong (X',D')$,
\item a line bundle $\xi\in \Pic^d(X)$, and
\item an automorphism $\rho\in \Aut(J(X))$ which is the identity on the $r$-torsion $J(X)[r]$
\end{itemize}
such that for each $(E,E_\bullet)\in \SM(X,D,r,\alpha,d)$ where $\Phi$ is defined $\sigma^*(E,E_\bullet)$ is obtained from $(E,E_\bullet)$ as a composition of the automorphism $\SA_\rho^\xi$ and a basic transformation $T$, i.e., a suitable composition of
\begin{itemize}
\item tensorization with a line bundle $L\in \Pic(X)$,
\item performing Hecke transformations at the parabolic points using the parabolic filtration, and
\item if $r>2$, taking the parabolic dual,
\end{itemize}
which sends bundles of degree $d$ to bundles of degree $d'$. Moreover, if $\Phi$ extends to an isomorphism, then the transformation $T$ is a correspondence between $\alpha$-stable parabolic vector bundles and $\sigma^*\alpha'$-stable parabolic vector bundles.
\end{theorem}

In particular, this description enables the computation of the groups of automorphisms and $3$-birational automorphisms of the moduli space (Theorem \ref{thm:automorphisms}). The idea of the proof is the following. We start by using the rationality of the moduli space with fixed determinant to recover the determinant morphism. Then the classification of $3$-birational maps of \cite{AG19} allows us to control the possible maps induced between the fibers of the determinant, reducing the problem to determining all ways in which ``basic transformations'' (pullback, tensorization, dualization and Hecke) can be glued along $\Pic^d(X)$ to produce a $3$-birational map on the whole moduli space. In order to do so, we study birationally the orbits of the action of $J(X)$ on generic points of the moduli. By using again the classification result from \cite{AG19} we can bound the possible images of such orbits, providing rigidity results for the structure of the global automorphisms and showing that, once restricted to a suitable open subset of the moduli space, they must be of the form described in the Theorem (composition of $\SA_\rho^\xi$ and a ``basic transformation''). Then we prove an extension result, showing that the obtained transformation actually agrees with $\Phi$ wherever it is defined and, if $\Phi$ is an isomorphism, providing an equality between the involved stability chambers.

This paper is structured as follows. Section \ref{section:preliminaries} introduces the basic notation and properties of the moduli spaces of parabolic vector bundles. Then, in Section \ref{section:basicTransformations} we introduce the group of ``basic transformations'' and some properties of its action on the moduli space. In Section \ref{section:Torelli} the Torelli theorem and $3$-birational Torelli theorem for the moduli space of parabolic vector bundles with fixed degree are proven. The action of the automorphisms of the Jacobian of the curve which fix the $r$ torsion on $\SM(X,D,r,\alpha,d)$ is described in Section \ref{section:action} and, through Section \ref{section:automorphisms}, we prove that this action, together with the ``basic transformations'', generates all isomorphisms and $3$-birational maps between moduli spaces of parabolic vector bundles with fixed degree. Finally, at the end of the section, a presentation of the automorphism group and $3$-birational automorphism group of the moduli space is provided.

\noindent\textbf{Acknowledgments.} 
This research was funded by MICINN (grants MTM2016-79400-P, PID2019-108936GB-C21 and ``Severo Ochoa Programme for Centres of Excellence in R\&D'' SEV-2015-0554). The author was also supported by a postdoctoral grant from ICMAT Severo Ochoa project. I would like to tank Tomás Gómez for helpful discussions.

\section{Moduli spaces of parabolic vector bundles}
\label{section:preliminaries}
Let $X$ be a smooth complex projective curve. Let $D\subset X$ be a finite set of different points. A parabolic vector bundle on $(X,D)$ is a vector bundle $E$ on $X$ together with a weighted flag on the fiber $E|_x$ over each $x\in D$
$$E|_x=E_{x,1} \supsetneq E_{x,2} \supsetneq \cdots \supsetneq E_{x,l_x} \supsetneq E_{x,l_{x}+1}=0$$
$$0\le \alpha_1(x) <\alpha_2(x) <\ldots <\alpha_{l_x}(x)<1$$
We call $\alpha=\{(\alpha_1(x),\ldots,\alpha_{l_x}(x))\}$ the system of weights. In this paper we will always assume that $\alpha$ is full flag, meaning that $l_x=r$ for all $x\in D$. If a system of weights is not provided, the filtered bundle $(E,E_\bullet)=(E,\{E_{x,i}\})$ is called a quasiparabolic vector bundle. Let
$$\pdeg(E,E_\bullet)=\deg(E)+\sum_{x\in D} \sum_{i=1}^r \alpha_i(x)$$
denote the parabolic degree of $(E,E_\bullet)$ with respect to the weight system $\alpha$. We say that a parabolic vector bundle $(E,E_\bullet)$ is $\alpha$-stable (respectively $\alpha$-semistable) if for any subbundle $F\subset E$ with the induced parabolic structure and weights we have
$$\frac{\pdeg(F,F_\bullet)}{\rk(F)} < \frac{\pdeg(E,E_\bullet)}{\rk(E)} \quad \text{(respectively } \le \text )$$

Let $\SM(X,D,r,\alpha,d)$ be the moduli space of $\alpha$-semistable full flag parabolic vector bundels on $(X,D)$ of rank $r$, degree $d$ and weight system $\alpha$. If the marked curve $(X,D)$ is clear from the context, we will omit it from the notation and write simply $\SM(r,\alpha,d)$. Similarly, given a line bundle $\xi\in \Pic^d(X)$, let $\SM(X,D,r,\alpha,\xi)\subset \SM(X,D,r,\alpha,d)$ denote the subvariety corresponding to parabolic bundles $(E,E_\bullet)$ such that $\det(E):=\wedge^rE\cong \xi$. Once again, we will also use the notation $\SM(r,\alpha,\xi)$ if the curve is clear from the context.

Through the rest of this work, we will assume, moreover, that $\alpha$ is a generic system of weights, meaning that that there are no integers $r'<r$ and $m$ and subsets $I(x)\subset \{1,\ldots,r\}$ of size $r'$ for $x\in D$ such that
$$r'\sum_{x\in D}\sum_{i=1}^r\alpha_i(x)-r\sum_{x\in D}\sum_{i\in I(x)} \alpha_i(x) =m$$
Under this condition, all $\alpha$-semistable parabolic bundles are $\alpha$-stable (c.f. \cite[Corollary 2.3]{AG18}), so $\SM(r,\alpha,d)$ and $\SM(r,\alpha,\xi)$ are smooth complex varieties and, moreover, $\SM(r,\alpha,\xi)$ is projective and rational \cite[Theorem 6.1]{BY99}.

Finally, $\SM(r,\alpha,d)$ admits a natural determinant map
$$\det:\SM(r,\alpha,d) \longrightarrow \Pic^d(X)$$
defined by $\det(E,E_\bullet):=\det(E)$. Clearly, $\SM(r,\alpha,\xi)=\det^{-1}(\xi)\subset \SM(r,\alpha,d)$.

\section{Basic transformations of quasiparabolic vector bundles}
\label{section:basicTransformations}
Let us recall the definition of the group of basic transformations described in \cite{AG19}, which acts on families of quasiparabolic bundles and induces isomorphisms between moduli spaces of parabolic vector bundles. For more details, see \cite[\S 5]{AG19}.

\begin{definition}
A basic transformation on $(X,D)$ is a tuple $T=(\sigma,s,L,H)$, where
\begin{itemize}
\item $\sigma$ is an automorphism $\sigma:X\longrightarrow X$ such that $\sigma(D)=D$,
\item $s\in \{1,-1\}$,
\item $L\in \Pic(X)$
\item $H$ is a divisor on $X$ with $0\le H\le (r-1)D$.
\end{itemize}
\end{definition}
Basic transformations act on quasiparabolic vector bundles as follows.
$$T(E,E_\bullet)=\left\{ \begin{array}{ll}
\sigma^*\left(L\otimes \SH_H(E,E_\bullet)\right) & s=1\\
\sigma^*\left(L\otimes \SH_H(E,E_\bullet)\right)^\vee & s=-1
\end{array}\right. $$
Where, $(E,E_\bullet)^\vee$ denotes the dual of the quasiparabolic vector bundle, i.e., the bundle $E^\vee$ with the filtration given by the annihilators $\op{ann}(E_{x,i})\subset E^\vee|_x$, and, given $H=\sum_{x\in D} h_x x$ with $0\le h_x<r$ for each $x\in D$, $\SH_H(E,E_\bullet)$ denotes the Hecke transformation of the quasiparabolic vector bundle $(E,E_\bullet)$ iterated $h_x$ times on each point $x\in D$. More precisely, given $x\in D$, construct $\SH_x(E,E_\bullet)=(H,H_\bullet)$ as follows. Let $H$ be the unique vector bundle which fits in the sequence
$$0\longrightarrow H \longrightarrow E \longrightarrow (E|_x/E_{x,2})\otimes \SO_x \longrightarrow 0$$
Restricting to the point $x$, we have a long exact sequence
$$0\longrightarrow E|_x/E_{x,2} \otimes \SO_X(-x)|_x \longrightarrow H|_x \stackrel{f}{\longrightarrow} E|_x \longrightarrow E|_x/E_{x,2} \longrightarrow 0$$
Then the full flag parabolic structure on $E|_x$ induces the following full flag filtration $H_\bullet$ on $H|_x$
$$H|_x=f^{-1}(E_{x,2}) \supsetneq \cdots \supsetneq f^{-1}(E_{x,r}) \supsetneq f^{-1}(0)= \frac{E|_x}{E_{x,2}} \otimes \SO_X(-x)|_x \supsetneq \frac{E_{x,2}}{E_{x,2}} \otimes \SO_X(-x)|_x=0$$
Finally, if $D=\{x_1,\ldots,x_n\}\subset X$ define
$$\SH_H= \SH_{x_1}^{h_{x_1}} \circ \cdots \circ \SH_{x_n}^{h_{x_n}}$$
Taking determinants of the previous expressions, we can check that $T$ acts on $\Pic(X)$ as follows.
$$T(\xi)=\sigma^*\left(L^r\otimes \xi(-H)\right)^s$$
In particular, if we compute the degree of both sides of the previous expressions, we can define for each $d\in \ZZ$
$$T(d)=s(\deg(L)+d-|H|)$$
Finally, $T$ acts on full flag systems of weights $\alpha$ on $(X,D)$ as follows.
$$T(\alpha)_i(x)=\left\{ \begin{array}{ll}
\SH_H(\alpha)_i(\sigma^{-1}(x)) & s=1\\
1-\SH_H(\alpha)_{r-i+1}(\sigma^{-1}(x))& s=-1
\end{array}\right. $$
and it can be proven that $(E,E_\bullet)$ is $\alpha$-stable if and only if $T(E,E_\bullet)$ is $T(\alpha)$-stable \cite[\S 5]{AG19}. Basic transformations form a group $\ST$ which has the following presentation.
\begin{lemma}[{\cite[Lemma 5.7]{AG19}}]
\label{lemma:compositionRules}
The group of basic transformations $\ST$ is generated by
\begin{itemize}
\item $\Sigma_\sigma=(\sigma,1,\SO_X,0)$
\item $\SD^+=(\id,1,\SO_X,0)=\id_\ST$
\item $\SD^-=(\id,-1,\SO_X,0)$
\item $\ST_L=(\id,1,L,0)$
\item $\SH_H=(\id,1,\SO_X,H)$
\end{itemize}
And we have the following composition rules
\begin{enumerate}
\item $\Sigma_\sigma \circ \Sigma_{\tau}=\Sigma_{\sigma\circ \tau}$
\item $\SD^s\circ \SD^t=\SD^{st}$
\item $\ST_L\circ \ST_M=\ST_{L\otimes M}$
\item If $0\le H_i\le (r-1)D$ for $i=1,2$ then
$$\SH_{H_1}\circ \SH_{H_2}=\ST_{L_{H_1+H_2}}\circ\SH_{H_1+H_2-L_{H_1+H_2}}$$
where, given a divisor $F=\sum_{x\in D} f_x x$, we define
$$L_F=\sum_{x\in D} \left \lfloor \frac{f_x}{r} \right\rfloor x$$
\item $\Sigma_\sigma \circ \SD^s = \SD^s \circ \Sigma_\sigma$
\item $\Sigma_\sigma \circ \ST_L=\ST_{\sigma^*L}\circ \Sigma_\sigma$
\item $\Sigma_\sigma \circ \SH_H=\SH_{\sigma^*H} \circ \Sigma_\sigma$
\item $\SD^-\circ \ST_L = \ST_{L^{-1}}\circ \SD^-$
\item $\SD^-\circ \SH_H= \ST_{\SO_X(D)}\circ \SH_{rD-H} \circ \SD^-$, for $H>0$
\item $\ST_L\circ \SH_H = \SH_H\circ \ST_L$ 
\end{enumerate}
\end{lemma}
We will also consider the following subgoups of $\ST$
\begin{itemize}
\item $\ST^+=\{(\sigma,s,L,H)\in \ST \,  |\, s=1\}$
\item $\ST^-=\{(\sigma,s,L,H)\in \ST \,  |\, s=-1\}$
\item For each $\xi \in \Pic(X)$, $\ST_\xi=\{T\in \ST \,  |\, T(\xi)\cong \xi\}$
\item For each $d\in \ZZ$, $\ST_d=\{T\in \ST \,  |\, T(d)=d\}$
\item Given a system of weights $\alpha$, $\ST_\alpha=\{T\in \ST\,  |\, T(\alpha)\text{ in the same stability chamber as } \alpha\}$
\end{itemize}
More generally, we will use notations like $\ST_{d.\alpha}$ to denote $\ST_d \cap \ST_\alpha$. Some of the main results from \cite{AG19} can be summarized in the following  equalities. For curves of genus at least $6$ and $r>2$
$$\ST_\xi \cap \ST_\alpha=:\ST_{\xi,\alpha}=\Aut(\SM(r,\alpha,\xi))\subset \Aut_{3-\op{bir}}(\SM(r,\alpha,\xi))=\ST_\xi$$
and for $r=2$
$$\ST_{\xi,\alpha}\cap \ST^+=:\ST_{\xi,\alpha}^+=\Aut(\SM(2,\alpha,\xi))\subset \Aut_{3-\op{bir}}(\SM(2,\alpha,\xi))=\ST_\xi^+:=\ST_\xi \cap \ST^+$$

\begin{lemma}
\label{lemma:reciprocalTorelli}
Let $(X,D)$ be a smooth complex projective curve of genus $g\ge 3$. Let $\alpha$ and $\alpha'$ be generic full flag systems of weights and $d,d'\in \ZZ$. Then there exists a basic transformation $T\in \ST^+$ which induces a $3$-birational equivalence between $\SM(r,\alpha,d)$ and $\SM(r,\alpha',d')$.
\end{lemma}

\begin{proof}
We will work analogously to \cite[Proposition 8.6]{AG19}. Write $d'-d=rm-k$ for some $0\le k<r$. Take a parabolic point $x\in D$ and consider the basic transformation $T=\ST_{\SO_X(mx)}\circ \SH_{kx}$. Then $T(d)=d'$ and $T$ induces an isomorphism
$$T:\SM(r,\alpha,d) \longrightarrow \SM(r,T(\alpha),d')$$
By \cite[Lemma 2.3]{AG19}, there exist open subsets $\SU\subset \SM(r,T(\alpha),d')$ and $\SU'\subset \SM(r,\alpha',d')$ which parameterize quasi-parabolic vector bundles which are both $T(\alpha)$-stable and $\alpha'$-stable. Thus, identifying $\SU$ and $\SU'$, $T:T^{-1}(\SU)\longrightarrow \SU\cong \SU'$ induces a $3$-birational map between $\SM(r,\alpha,d)$ and $\SM(r,\alpha',d')$.
\end{proof}

\begin{lemma}
\label{lemma:basicTransFinite}
For every $\xi\in \Pic^d(X)$, $\ST_\xi$ is finite.
\end{lemma}

\begin{proof}
Basic transformations in $\ST_\xi$ are given by tuples $T=(\sigma,s,L,H)$ such that
$$\xi \cong T(\xi)=\sigma^*(L^r\otimes \xi(-H))^s$$
Since there are only a finite number of automorphisms of the curve $X$, $s=\pm 1$ and there is a finite number of choices of a divisor $H$ with $0\le H\le (r-1)D$, then it is enough to prove that for each choice of $\sigma$, $s$ and $H$ there is a finite number of line bundles $L$ with $(\sigma,s,L,H)\in \ST_\xi$. Let $L,L'\in \Pic(X)$ such that $(\sigma,s,L,H),(\sigma,s,L',H)\in \ST_\xi$. Then
$$\sigma^*(L^r\otimes \xi(-H))^s\cong \sigma^*((L')^r\otimes \xi(-H))^s$$
so, taking the pullback by $\sigma^{-1}$, rising the result to the power $s$ and tensoring with $\xi^{-1}(H)$ yields $L^r\cong (L')^r$. This implies that $L'=L\otimes S$ for some $S\in J(X)[r]$. As the $r$-torsion of the Jacobian is finite, then so is $\ST_\xi$.
\end{proof}

\begin{lemma}
\label{lemma:basicFaithful}
Let $(X,D)$ be a smooth complex projective curve of genus $g\ge 4$. Let $\xi\in \Pic^d(X)$ be any line bundle. For a generic $(E,E_\bullet)\in \SM(r,\alpha,\xi)$ the following is satisfied. If $r>2$ then for each basic transformation $T\in \ST$ such that $T\ne \id$, $T(E,E_\bullet)\not\cong (E,E_\bullet)$. If $r=2$ the same result is true for nontrivial transformations $T\in \ST^+$.
\end{lemma}

\begin{proof}
By \cite[Lemma 7.24]{AG19}, given a fixed $\xi$ and $r>2$, for each $T\in \ST$, there exists an open subset $\SU_T\subset \SM(r,\alpha,\xi)$ such that for each $(E,E_\bullet)\in \SU_T$, $T(E,E_\bullet)\not\cong (E,E_\bullet)$. Let us consider the subset
$$\SU=\bigcap_{T\in \ST_\xi} \SU_T$$
by Lemma \ref{lemma:basicTransFinite}, it is an intersection of a finite number of open dense subsets, so it is open and dense. Let us verify that the parabolic bundles in $\SU$ are not fixed by any basic transformation. Suppose the contrary. Let $T\in \ST$ and $(E,E_\bullet)\in \SU$ such that $T(E,E_\bullet)\cong (E,E_\bullet)$. Then, taking determinants, we have
$$\xi=\det(E)\cong \det(T(E,E_\bullet))=T(\xi)$$
so $T\in \ST_\xi$, but this is impossible, since this implies that $(E,E_\bullet)\in \SU\subset \SU_T$.
The case $r=2$ is completely analogous, building $\SU_T\subset \SM(2,\alpha,\xi)$ for $T\in \ST^+$ via \cite[Lemma 7.24]{AG19} and then taking $\SU=\bigcap_{T\in \ST^+_\xi} \SU_T$, where $\ST^+_\xi=\ST^+\cap \ST_\xi$.
\end{proof}

\section{Torelli type theorems}
\label{section:Torelli}

Before engaging the full classification, we will start by proving a Torelli type theorem, stating that the $3$-birational class of the moduli space $\SM(X,D,r,\alpha,d)$ identifies univocally the isomorphism class of the marked curve $(X,D)$. We will start by proving  that we can recover the determinant map.

\begin{lemma}
\label{lemma:recoverDeterminant}
Let $\alpha$ and $\alpha'$ be full flag systems of weights over marked curves $(X,D)$ and $(X',D')$ respectively. Let $\Phi:\SM(X,D,r,\alpha,d)\dashrightarrow \SM(X',D',r',\alpha',\xi')$ be a birational map. Then there is an isomorphism $\varphi:\Pic^d(X) \longrightarrow \Pic^{d'}(X')$ such that the following diagram commutes.
\begin{eqnarray*}
\xymatrixcolsep{3.5pc}
\xymatrix{
\SM(X,D,r,\alpha,d) \ar@{-->}[r]^{\Phi} \ar[d]_{\det} & \SM(X',D',r',\alpha',d') \ar[d]^{\det}\\
\Pic^d(X) \ar[r]^{\varphi} & \Pic^{d'}(X')
}
\end{eqnarray*}
\end{lemma}

\begin{proof}
Let $\SU\subset \SM(X,D,r,\alpha,d)$ and $\SU'\subset \SM(X',D',r',\alpha',d')$ be open subsets such that $\Phi:\SU\to \SU'$ is an isomorphism. Let $\SZ$ and $\SZ'$ be the complements of $\SU$ and $\SU'$ respectively. As $\dim(\SZ)\le \dim(\SM(X,D,r,\alpha,d))-1$, then the generic fiber of the map $\det:\SZ\to \Pic^d(X)$ must have dimension at most $\dim(\SM(X,D,r,\alpha,d))-\dim(\Pic^d(X))-1=\dim(\SM(X,D,r,\alpha,\xi))-1$. The dimension of the fiber is upper semicontinuous, so there is an open dense subset $\SP \subset \Pic^d(X)$ such that for every $\xi\in \SP$, $\dim(\SZ\cap \det^{-1}(\xi))\le \dim(\det^{-1}(\xi))-1$. Therefore, for each $\xi \in \SP$, $\det^{-1}(\xi)\cap \SU$ is dense in $\SM(X,D,r,\alpha,\xi)$ and, therefore, $\det^{-1}(\xi)\cap \SU$ is rational.

Let $\tilde{\SU}=\det^{-1}(\SP)\cap \SU$. It is an open dense subset of $\SM(X,D,r,\alpha,d)$ for which $\Phi$ is well defined. Therefore, by composition, it admits a map $\tilde{\SU} \to \SU' \to \Pic^{d'}(X')$. As the fibers of $\tilde{\SU}\longrightarrow \SP$ for the determinant map are rational, each fiber must map to a single point in $\Pic^{d'}(X')$, so this map descends to a map $\varphi:\SP\to \Pic^{d'}(X')$. As the map $\det:\SM(X',D',r',\alpha',d')\to \Pic^{d'}(X')$ is surjective with equidimensional fibers and $\Phi(\tilde{\SU})\subset \SU'$ is dense, then its image $\varphi(\SP)$ is dense in $\Pic^{d'}(X')$. Repeating the argument for $\Phi^{-1}$ proves that there are open dense subsets $\tilde{\SP}\subset \Pic^{d}(X)$ and $\tilde{\SP}'\subset \Pic^{d'}(X')$ such that $\varphi:\tilde{\SP} \to \tilde{\SP}'$ is an isomorphism, thus inducing a birational map $\Pic^d(X) \dashrightarrow \Pic^{d'}(X')$. Birational maps between abelian varieties extend uniquely to isomorphisms (c.f. \cite[Theorem 3.8]{Milne08}), so $\varphi$ extends to an isomorphism $\varphi:\Pic^d(X)\to \Pic^{d'}(X')$. It only remains to prove that the diagram above is still commutative wherever $\Phi$ is defined. $\det \circ \Phi$ and $\varphi\circ \det$ are well defined maps $\SU\longrightarrow \Pic^{d'}(X)$ which agree on $\tilde{\SU}$. As $\tilde{\SU}$ is dense and $\Pic^{d'}(X)$ is separable, they are equal on $\SU$, so the diagram of rational maps is commutative.
\end{proof}

\begin{corollary}
\label{cor:openSubsets}
Let $\alpha$ and $\alpha'$ be full flag systems of weights over marked curves $(X,D)$ and $(X',D')$ respectively. Let $\Phi:\SM(X,D,r,\alpha,d)\dashrightarrow \SM(X',D',r',\alpha',\xi')$ be a $3$-birational map. Let $\varphi:\Pic^d(X) \longrightarrow \Pic^{d'}(X')$ be the isomorphism given by the previous lemma. Then there exist open dense subsets $\SU\subset \SM(X,D,r,\alpha,d)$, $\SU'\subset \SM(X',D',r',\alpha',d')$, $\SP\subset \Pic^d(X)$ and $\SP'\subset \Pic^{d'}(X')$ such that
\begin{itemize}
\item $\Phi$ is defined on $\SV$
\item The following diagram commutes
\begin{eqnarray*}
\xymatrixcolsep{3pc}
\xymatrixrowsep{2.5pc}
\xymatrix@!0{
\SM(X,D,r,\alpha,d) \ar@{-->}[rrrr]^{\Phi} \ar[ddd]_{\det} &&&& \SM(X',D',r',\alpha',d') \ar[ddd]^{\det}\\
&\SV \ar[rr]^{\Phi} \ar[d]_{\det} \ar@{_(->}[ul] && \SV' \ar[d]^{\det} \ar@{^(->}[ur]  &\\
& \SP \ar[rr]^{\varphi} \ar@{_(->}[dl] && \SP' \ar@{^(->}[dr] & \\
\Pic^d(X) \ar[rrrr]^{\varphi} &&&& \Pic^{d'}(X')
}
\end{eqnarray*}
\item For each $\xi\in \SP$, let $\SV_\xi=\SU\cap \det^{-1}(\xi)$. Let $\xi'=\varphi(\xi)$ and $\SU_{\xi'}'=\Phi(\SV_\xi)$. Then
$$\codim(\SM(X,D,r,\alpha,\xi)\backslash \SV_\xi)\ge 3, \quad \quad \codim(\SM(X',D',r',\alpha',\xi')\backslash \SV'_{\xi'})\ge 3$$
\end{itemize}
\end{corollary}

\begin{proof}
By the previous Lemma we know that $\dim(Pic^d(X))=\dim(\Pic^{d'}(X'))$ and $\dim(\SM(X,D,r,\alpha,d))=\dim(\SM(X',D',r',\alpha',d'))$ so, as the determinant map is equidimensional, we know that the dimension of their respective fibers is the same.

As before, let $\SU\subset \SM(X,D,r,\alpha,d)$ and $\SU'\subset \SM(X',D',r',\alpha',d')$ be open subsets such that there is an isomorphism $\Phi:\SU\to \SU'$ and let $\SZ$ and $\SZ'$ be their respective complements in the moduli spaces.
We know that $\codim(\SZ)\ge 3$, and the map $\det:\SM(X,r,\alpha,d)\to \Pic^d(X)$ is equidimensional and surjective, so for a generic $\xi\in \Pic^{d}(X)$, the codimension of $\SZ\cap \det^{-1}(\xi)$ in $\SM(X,r,\alpha,\xi)$ is at least $3$. Let $\SP_1\subset \Pic^d(X)$ be the open subset where this happens. Similarly, for a generic $\xi'\in \Pic^{d'}(X')$, the codimension of $\SZ\cap \det^{-1}(\xi')$ in $\SM(X',r',\alpha',\xi')$ is at least $3$. Let $\SP_1'\subset \Pic^{d'}(X)$ be the open subset of such line bundles. Finally, take
$$\SP=\SP_1\cap \varphi^{-1}(\SP_1'),\quad \quad \SP'=\varphi(\SP),$$
$$\SV=\SU\cap {\det}^{-1}(\SP),\quad \quad \SV'=\varphi(\SV).$$
\end{proof}

\begin{theorem}[{$3$-birational Torelli theorem}]
\label{thm:TorelliBir}
Let $(X,D)$ and $(X',D')$ be two smooth complex projective curves of genus $g\ge 4$ and $g'\ge 4$ respectively with sets of marked points $D\subset X$ and $D'\subset X'$. Let $\alpha$ and $\alpha'$ be full flag generic systems of weights over $(X,D)$ and $(X',D')$ respectively. Then $\SM(X,D,r,\alpha,d)$ and $\SM(X',D',r',\alpha',d')$ are $3$-birational if and only if
\begin{enumerate}
\item $r=r'$
\item $(X,D)$ is isomorphic to $(X',D')$, i.e., there exists an isomorphism $\sigma:X\stackrel{\sim}{\to} X'$ sending $D$ to $D'$.
\end{enumerate}
\end{theorem}

\begin{proof}
By Corollary \ref{cor:openSubsets}, there exists $\xi\in \Pic^d(X)$ and $\xi'=\varphi(\xi)\in \Pic^{d'}(X')$ and open subsets $\SV_\xi \subset \SM(X,D,r,\alpha,\xi)$, $\SV'_{\xi'} \subset \SM(X',D',r',\alpha',\xi')$such that
$$\Phi(\SV_\xi)=\SV'_{\xi'}$$
$$\codim(\SM(X,D,r,\alpha,\xi)\backslash \SU_\xi, \SM(X,D,r,\alpha,\xi))\ge 3$$
$$\codim(\SM(X',D',r',\alpha',\xi')\backslash \SU'_{\xi'}, \SM(X',D',r',\alpha',\xi'))\ge 3$$
Thus, $\Phi$ induces a $3$-birational equivalence between $\SM(X,D,r,\alpha,\xi)$ and $\SM(X',D',r',\alpha',\xi')$. Now we can apply the $3$-birational version of the Torelli theorem for parabolic vector bundles \cite[Theorem 8.5]{AG19} to obtain the result. The reciprocal is a direct consequence of Lemma \ref{lemma:reciprocalTorelli}.
\end{proof}

In particular, this proves that the isomorphism class of the moduli space of parabolic bundles with fixed degree identifies uniquely the isomorphism class of the marked curve $(X,D)$.

\begin{corollary}[Torelli theorem]
\label{cor:Torelli}
Let $(X,D)$ and $(X',D')$ be two smooth complex projective curves of genus $g\ge 4$ and $g'\ge 4$ respectively with set of marked points $D\subset X$ and $D'\subset X'$. Let $\alpha$ and $\alpha'$ be full flag generic systems of weights over $(X,D)$ and $(X',D')$ respectively. If $\SM(X,D,r,\alpha,d)$ and $\SM(X',D',r',\alpha',d')$ are isomorphic then
\begin{enumerate}
\item $r=r'$
\item $(X,D)$ is isomorphic to $(X',D')$, i.e., there exists an isomorphism $\sigma:X\stackrel{\sim}{\to} X'$ sending $D$ to $D'$.
\end{enumerate}
\end{corollary}

This Corollary and the previous Theorem provide strong refinements of the Torelli type theorem obtained in \cite{BGL16} when the parabolic weights are generic. In that paper, it was proven that any isomorphism $\SM(X,D,r,\alpha,d)\stackrel{\sim}{\longrightarrow} \SM(X',D',r,\alpha,d)$ which preserves the Néron-Severi class of the determinant bundle of the moduli space induces an isomorphism $(X,D)\cong (X',D')$. In this case, the generalization is twofold. On one hand, Theorem \ref{thm:TorelliBir} proves a Torelli theorem between moduli spaces which are built with possibly different choices of the rank $r$, degree $d$ and, more importantly, the stability condition $\alpha$. On the other hand, it proves that the assumption on the determinant bundle is not necessary and, in fact, it shows a correspondence between $3$-birational classes of moduli spaces of parabolic bundles and pairs consisting on the isomorphism class of marked curves $(X,D)$ and a rank $r$.

\section{Action of the automorphisms of the Jacobian}
\label{section:action}

Let $J(X)[r]$ denote the $r$-torsion subgroup of the Jacobian $J(X)$. Let $\Aut(J(X),J(X)[r])$ denote the group of automorphisms of the Jacobian of $X$ which fix the $r$-torsion subgroup, i.e., the group of automorphisms $\rho:J(X)\longrightarrow J(X)$ such that $\rho|_{J(X)[r]}=\id_{J(X)[r]}$. Observe that, as $J(X)$ is abelian, all automorphisms of the Jacobian factor as the composition of a group automorphism and a translation (c.f. \cite[Corollary 1.2]{Milne08}). By fixing the subset $J(X)[r]$, which includes $\SO_X\in J(X)[r]$, all elements in $\Aut(J(X),J(X)[r])$ are, in particular, group automorphisms.

In this section we will describe how this group $\Aut(J(X),J(X)[r])$ acts by automorphisms on the moduli space of parabolic vector bundles $\SM(r,\alpha,d)$. Let $\rho\in \Aut(J(X),J(X)[r])$. Then $\rho-\id:J(X)\longrightarrow J(X)$  is a group homomorphism such that  $(\rho-\id)|_{J(X)[r]}=0$. Therefore, $(\rho-\id)$ factorizes through the multiplication by $r$, and there exists a unique homomorphism $\wt{\rho}:J(X)\longrightarrow J(X)$ such that
\begin{eqnarray*}
\xymatrix{
J(X)[r] \ar[dr]^0 \ar@{^(->}[d] &\\
J(X) \ar[r]^{\rho-\id} \ar@{->>}[d]_{r \cdot(-)} & J(X)\\
J(X) \ar@{-->}[ru]_{\wt{\rho}} &\\
}
\end{eqnarray*}
As $f$ is an homomorphism, $f(rx)=rf(x)$ for all $x\in J(X)$, so $\rho=\id+r\wt{\rho}$. A direct computation allows us to verify the following property of $\wt{\rho}$.

\begin{lemma}
\label{lemma:tildeComposition}
Let $\rho,\rho'\in \Aut(J(X),J(X)[r])$. Then
$$\wt{\rho'\circ \rho}=\wt{\rho'}+\wt{\rho}+r\wt{\rho'}\circ\wt{\rho}$$
\end{lemma}

\begin{proof}
Since $\rho=\id+r\wt{\rho}$ and $\rho'=\id+r\wt{\rho'}$, then
$$\rho'\circ \rho = (\id+r\wt{\rho'})\circ (\id+r\wt{\rho}) = \id+ r\wt{\rho'}+r\wt{\rho}+r^2\wt{\rho'}\circ \wt{\rho}=\id+r\left( \wt{\rho'}+\wt{\rho}+r\wt{\rho'}\circ \wt{\rho}\right)$$
As $\rho'\circ \rho = \id + r \wt{\rho'\circ \rho}$ and $\wt{\rho'\circ\rho}$ is unique, the equality follows. 
\end{proof}

We can then define an action of $\Aut(J(X),J(X)[r])$ on $\SM(r,\alpha,d)$ as follows. Let us fix once an for all a line bundle $\xi\in \Pic^d(X)$. Then for each $\rho\in \Aut(J(X),J(X)[r])$, define
$$\rho\cdot (E,E_\bullet) := (E,E_\bullet) \otimes \wt{\rho}(\det(E)\otimes \xi^{-1})$$

This map is clearly well defined on families of parabolic vector bundles, and tensoring by an element of the Jacobian preserves the stability of the parabolic vector bundle and the degree of the underlying bundle, so it defines a map
$$\SA^\xi: \Aut(J(X),J(X)[r])\times \SM(r,\alpha,d) \longrightarrow \SM(r,\alpha,d)$$
For each $\rho\in \Aut(J(X),J(X)[r])$, let $\SA_\rho^\xi:\SM(r,\alpha,d)\longrightarrow \SM(r,\alpha,d)$ denote the induced self-map of the moduli space.

\begin{lemma}
\label{lemma:autoJacobianAction}
The map $\SA^\xi: \Aut(J(X),J(X)[r]) \times \SM(r,\alpha,d)\longrightarrow \SM(r,\alpha,d)$ is an action by automorphisms of the moduli space.
\end{lemma}

\begin{proof}
First of all, observe that, by construction, $\wt{\id}=0$ implies that $\SA_{\id}^\xi=\id_{\SM(r,\alpha,d)}$. Thus, if we prove that the map $\SA^\xi$ is an action, then for each $\rho\in \Aut(J(X),J(X)[r])$, we have
$$\SA_\rho^\xi \circ \SA_{\rho^{-1}}^\xi = \SA_{\rho\circ\rho^{-1}}^\xi=\SA_{\id}^\xi=\id_{\SM(r,\alpha,d)}=\SA_{\rho^{-1}}^\xi \circ \SA_{\rho}^\xi$$
so $\SA_\rho^\xi$ is an invertible map and, therefore, an isomorphism.

It is, therefore, enough to check that $\SA_\rho^\xi \circ \SA_{\rho'}^\xi=\SA_{\rho\circ \rho'}^\xi$ for any $\rho,\rho'\in \Aut(J(X),J(X)[r])$. Let $E\in \SM(r,\alpha,\xi)$. Let 
$$(E',E'_\bullet)=\SA_\rho(E,E_\bullet)^\xi=(E,E_\bullet) \otimes \wt{\rho}(\det(E)\otimes \xi^{-1})$$
Then
$$\det(E')=\det(E)\otimes \wt{\rho}(\det(E)\otimes \xi^{-1})^r = \det(E) \otimes (\rho-\id)(\det(E)\otimes \xi^{-1}) = \rho(\det(E)\otimes \xi^{-1})\otimes \xi$$
Thus, taking into account the relation from Lemma \ref{lemma:tildeComposition},
\begin{multline*}
\SA_\rho'(\SA_\rho(E,E_\bullet))^\xi=(E',E'_\bullet)\otimes \wt{\rho'}(\det(E')\otimes \xi^{-1})\\
=(E,E_\bullet)\otimes \wt{\rho}(\det(E)\otimes \xi^{-1}) \otimes \wt{\rho'}\left(\rho(\det(E)\otimes \xi^{-1}\right)\\
=(E,E_\bullet)\otimes \wt{\rho}(\det(E)\otimes \xi^{-1}) \otimes \wt{\rho'}\left(\det(E)\otimes \xi^{-1}+r\wt{\rho}(\det(E)\otimes \xi^{-1})\right)\\
=(E,E_\bullet)\otimes \left( \wt{\rho}(\det(E)\otimes \xi^{-1})+\wt{\rho'}(\det(E)\otimes \xi^{-1}) + r \wt{\rho'}\circ \wt{\rho}(\det(E)\otimes \xi^{-1}) \right)\\
= (E,E_\bullet) \otimes \wt{\rho'\circ \rho}(\det(E)\otimes \xi^{-1}) = \SA_{\rho'\circ\rho}^\xi(E,E_\bullet)
\end{multline*}
\end{proof}

\begin{remark}
\label{rmk:changexi}
The action $\SA^\xi$ depends on the choice of the line bundle $\xi\in \Pic^d(X)$. Nevertheless, given a fixed $\rho\in \Aut(J(X),J(X)[r])$, different choices of $\xi$ give rise to automorphisms $\SA^\xi_\rho$ which only differ by tensorization by an element in $J(X)$. More precisely, given $\xi,\xi'\in \Pic^d(X)$, $L=\xi'\otimes \xi^{-1}\in J(X)$ satisfies that for each $(E,E_\bullet)\in \SM(r,\alpha,d)$ 
$$\SA_\rho^{\xi}(E,E_\bullet)=(E,E_\bullet)\otimes \wt{\rho}(\det(E)\otimes \xi^{-1}) = (E,E_\bullet)\otimes \wt{\rho}(\det(E) \otimes (\xi')^{-1} \otimes L) = \SA_\rho^{\xi'}(E,E_\bullet) \otimes \wt{\rho}(L)$$
\end{remark}

\section{3-birational maps between moduli spaces of parabolic bundles}
\label{section:automorphisms}

Now, we can engage the classification of $3$-birational maps between moduli of parabolic bundles of fixed degree. First of all, we can use the classification in the fixed determinant case to obtain the following.

\begin{lemma}
\label{lemma:basicTransFibres}
Let $(X,D)$ be a smooth complex projective curve of genus $g\ge 6$. Let $\Phi:\SM(r,\alpha,d)\dashrightarrow \SM(r,\alpha,d)$ be a $3$-birational map. Let $\SP\subset \Pic^d(X)$ be the open subset given by Corollary \ref{cor:openSubsets}. Then for each $\xi\in \SP$, there exists a basic transformation $T_\xi\in \ST_d$ such that
$$\Phi|_{\SM(r,\alpha,\xi)} = T_\xi$$
moreover, if $r=2$, then we can take $T\in \ST^+_d$.
\end{lemma}

\begin{proof}
By Corollary \ref{cor:openSubsets}, there exist an isomorphism $\varphi: \Pic^d(X) \longrightarrow \Pic^d(X)$, an open subset $\SP\subset \Pic^d(X)$ and, for each $\xi\in \SP$, open subsets $\SV_\xi\subset \SM(r,\alpha,\xi)$ and $\SV'_{\varphi(\xi)}\subset \SM(r,\alpha,\varphi(\xi))$ such that for each $\xi\in \SP$
\begin{itemize}
\item $\Phi$ is defined over $\SV_\xi$
\item $\Phi(\SV_\xi)=\SV'_{\varphi(\xi)}$
\item The complements of $\SV_\xi$ and $\SV_{\varphi(\xi)}'$ in $\SM(r,\alpha,\xi)$ and $\SM(r,\alpha,\varphi(\xi))$ respectively have codimension is at least $3$.
\end{itemize}
Thus, $\Phi|_{\SV_\xi}:\SM(r,\alpha,\xi) \dashrightarrow \SM(r,\alpha,\varphi(\xi))$ is a $3$-birational map. Using \cite[Theorem 8.10]{AG19}, this map must be given by a basic transformation $T\in \ST$ which sends $T(\xi) \cong \varphi(\xi)$. As $\deg(\varphi(\xi))=\deg(\xi)=d$, then we conclude that $T\in \ST_d$.
\end{proof}

Given a point $(E,E_\bullet)\in \SM(r,\alpha,d)$, let $J(X)\cdot (E,E_\bullet)$ denote the orbit of $(E,E_\bullet)$ by the action of $J(X)\subset \ST_d$. Let us study its geometry and the image of this type of orbits  under $3$-birational self-maps of the moduli space.

\begin{lemma}
\label{lemma:JacobianOrbit}
Let $(X,D)$ be a smooth complex projective curve of genus $g\ge 4$ and let $\xi\in \Pic^d(X)$. For a generic parabolic vector bundle $(E,E_\bullet)\in \SM(r,\alpha,\xi)$ we have the following. For all basic transformations $T\in \ST_d$ the map
$$J_T: J(X) \longrightarrow \SM(r,\alpha,d)$$
given by
$$J_T(L)=T(E,E_\bullet) \otimes L$$
is well defined and it gives a bijection $J_T:J(X) \longrightarrow J(X) \cdot T(E,E_\bullet)$ such that there exists an open subset $U_T\subset J(X)$ for which
$$J_T: U_T \stackrel{\sim}{\longrightarrow} U_T\cdot T(E,E_\bullet)$$
is an isomorphism.
\end{lemma}

\begin{proof}
Let $\SU\subset \SM(r,\alpha,\xi)$ be the open subset of parabolic vector bundles which are not fixed by any $T\in \ST$ (or $T\in \ST^+$ if $r=2$) given by Lemma \ref{lemma:basicFaithful}. Let $\SM^{\op{us}}(r,\xi) \subset \SM(r,\alpha,\xi)$ be the open subset corresponding to parabolic vector bundles which are $\alpha'$-stable for every generic system of weights $\alpha'$ on $(X,D)$. This open subset exists by \cite[Corollary 10.4]{AG19}. Take $\SU'=\SU\cap \SM^{\op{us}}(r,\xi)$. Then for each $(E,E_\bullet)\in \SU'$ and each $T\in \ST_d$, $(E,E_\bullet)$ is $T^{-1}(\alpha)$-stable, so $T(E,E_\bullet)\in \SM(r,\alpha,d)$. Moreover, $T(E,E_\bullet)\not\cong (E,E_\bullet)$ for all $T\in \ST$ if $r>2$ or for all $T\in \ST^+$ if $r=2$. Then let us prove that for each $T\in \ST_d$ and $L,L'\in J(X)$ with $L\ne L'$,
$$T(E,E_\bullet)\otimes L\not\cong T(E,E_\bullet)\otimes L'$$
Assume the opposite. Suppose that $T(E,E_\bullet)\otimes L \cong T(E,E_\bullet)\otimes L'$. Tensoring with $(L')^{-1}$ and composing with $T^{-1}$ yields
$$(T^{-1} \circ \ST_{L\otimes (L')^{-1}} \circ T)(E,E_\bullet) \cong (E,E_\bullet)$$
independently on the choice of $T\in \ST_d$, it is clear that $T^{-1} \circ \ST_{L\otimes (L')^{-1}} \circ T\in \ST_d^+$, and, by Lemma \ref{lemma:basicFaithful}, $(E,E_\bullet)$ is not fixed by any nontrivial such transformation, so $T^{-1}\circ \ST_{L\otimes (L')^{-1}} \circ T=\id$, but this is impossible, as composing with $T$ on the left and $T^{-1}$ on the right yields $\ST_{L\otimes (L')^{-1}}=\id$, which implies $L\otimes (L')^{-1}=\SO_X$, i.e., $L=L'$. Therefore, the map $J_T:J(X)\longrightarrow \SM(r,\alpha,d)$ is a bijection with its image, which is $J(X)\cdot T(E,E_\bullet)$ by construction.

Let $(J(X)\cdot T(E,E_\bullet))^{\op{sm}}$ denote the smooth part of the variety $J(X)\cdot T(E,E_\bullet)$. Let
$$U_T=J_T^{-1}\left((J(X)\cdot T(E,E_\bullet))^{\op{sm}} \right)$$
since $J_T$ is a bijection,
$$J_T(U_T)=U_T\cdot T(E,E_\bullet)=(J(X)\cdot T(E,E_\bullet))^{\op{sm}}$$
by construction, $J_T$ gives a bijective algebraic morphism between two smooth complex algebraic varieties, $U_T$ and $J_T(U_T)$. As a consequence of Zariski's main Theorem \cite[IV Corollary 8.12.13]{EGAIV}, it must be an isomorphism between them (c.f. \cite[Remark 8.21]{Milne09}).
\end{proof}

\begin{lemma}
Suppose that $g\ge 2$. Let $(E,E_\bullet)$ be a generic point in $\SM(r,\alpha,d)$. Then there exists a finite subset $R\subset \ST_d$ such that
$$\bigcup_{T\in \ST_d} J(X) \cdot T(E,E_\bullet)=\coprod_{T\in R} J(X) \cdot T(E,E_\bullet)$$
\end{lemma}

\begin{proof}
Let
$$S=\{T=(\sigma,s,\SO_X,H)\in \ST \, | \, T(d)=s(d-|H|) \cong d \pmod{r}\}\subset \ST$$
As there is only a finite number of choices for $\sigma$, $s$ and $H$, $S$ is a finite set. For each $(\sigma,s,\SO_X,H)\in S$, fix a line bundle $L_{\sigma,s,H}$ of degree $(d-s(d-|H|))/r$ (which is an integer by construction). A direct computation then shows that $\ST_L\circ (\sigma,s,\SO_X,H)\in \ST_d$. Let us consider the finite subset
$$S'=\{ \ST_{L_{\sigma,s,H}}\circ (\sigma,s,\SO_X,H) \, |\,  (\sigma,s,\SO_X,H)\in S\} \subset \ST_d$$
Now, let $T\in \ST_d$ be any basic transformation which preserves the degree $d$. Using the relations from Lemma \ref{lemma:compositionRules}, we can write $T$ as $T=\ST_L\circ (\sigma,s,\SO_X,H)$ for some $L\in \Pic(X)$, $\sigma\in \Aut(X)$, $s=\pm 1$ and $0\le H\le (r-1)D$. As
$$d=(\ST_L\circ (\sigma,s,\SO_X,H))(d)=r\deg(L)+s(d-|H|) \cong s(d-|H|) \pmod{r}$$
we have that $(\sigma,s,\SO_X,H)\in S$. Moreover, $\deg(L)=s(d-|H|)/r=\deg(L_{\sigma,s,H})$ so there exists $F\in J(X)$  such that $L=F\otimes L_{\sigma,s,H}$. Thus
$$T(E,E_\bullet)=F\otimes (\ST_{L_{\sigma,s,H}}\circ (\sigma,s,\SO_X,H))(E,E_\bullet) \in J(X) \cdot (\ST_{L_{\sigma,s,H}} \circ (\sigma,s,\SO_X,H))(E,E_\bullet)$$
Therefore, we conclude that
$$\bigcup_{T\in \ST_d} J(X) \cdot T(E,E_\bullet)=\bigcup_{T\in S'} J(X) \cdot T(E,E_\bullet)$$
Finally, observe that the right hand side of the previous equality is a finite union of orbits of $J(X)$ in $\SM(r,\alpha,d)$. As any two orbits either coincide or are disjoint, then there is a subset $R\subset S'$, which is also finite and satisfies that
$$\bigcup_{T\in S'} J(X) \cdot T(E,E_\bullet)=\coprod_{T\in R} J(X) \cdot T(E,E_\bullet)$$
\end{proof}

\begin{lemma}
\label{lemma:JacobianOrbitsFixed}
Let $(X,D)$ be a smooth complex projective curve of genus $g\ge 6$. Let $\Phi:\SM(r,\alpha,d)\dashrightarrow \SM(r,\alpha,d)$ be a $3$-birational map. For a generic $\xi\in \Pic^d(X)$ and a generic $(E,E_\bullet)\in \SM(r,\alpha,\xi)$ there exist open subsets $\tilde{\SP},\tilde{\SP}'\subset J(X)$ and a basic transformation $T\in \ST_d$ such that
\begin{itemize}
\item $J_{\id}$ and $J_T$ induce isomorphisms $\tilde{\SP}\stackrel{J_{\id}}{\cong} \tilde{\SP}\cdot (E,E_\bullet)$, $\tilde{\SP}'\stackrel{J_T}{\cong} \tilde{\SP}'\cdot T(E,E_\bullet)$,
\item $\Phi$ is defined over $\tilde{\SP}\cdot (E,E_\bullet)$,
\item $\Phi(\tilde{\SP}\cdot (E,E_\bullet))=\tilde{\SP}'\cdot T(E,E_\bullet)$, and
\item if $r=2$, $T\in \ST_d^+$.
\end{itemize}
\end{lemma}

\begin{proof}
Let $\SU,\SU' \subset \SM(r,\alpha,d)$ and $\SP,\SP'\subset \Pic^d(X)$ be the open subsets given by Corollary \ref{cor:openSubsets}. Let $\xi\in \SP$ and let $(E,E_\bullet)\in \SM(r,\alpha,\xi)\cap \SU$ be a generic element in the sense of Lemma \ref{lemma:JacobianOrbit}. By Lemma \ref{lemma:JacobianOrbit}, there exists an open subset $U=U_{\id}\subset J(X)$ such that $U\cong J_{\id}(U)\subset J(X)\cdot (E,E_\bullet)$. Let $V=U\cap J_{\id}^{-1}(\SU)$. Then $\Phi$ is well defined on $V\cdot (E,E_\bullet)$. Let us start by proving that
\begin{equation}
\label{eq:unionOrbits}
\Phi(V \cdot (E,E_\bullet)) \subset \bigcup_{T\in \ST_d} J(X)\cdot T(E,E_\bullet)
\end{equation}
By construction, for each $L\in V$, $\xi\otimes L^r=\det((E,E_\bullet)\otimes L) \in \SP$. Thus, by Lemma \ref{lemma:basicTransFibres} there exists a basic transformation $T_{\xi\otimes L^r}\in \ST_d$ such that
$$\Phi((E,E_\bullet)\otimes L)\cong T_{\xi\otimes L^r}((E,E_\bullet)\otimes L)$$
On the other hand, for each $T\in \ST_d$,
$$T((E,E_\bullet)\otimes L) = (T\circ \ST_L)(E,E_\bullet)$$
As $L\in J(X)$, then $T'=(T\circ \ST_L)\in \ST_d$, so $T((E,E_\bullet)\otimes L)=T'(E,E_\bullet)\in J(X)\cdot T'(E,E_\bullet)$. If $r=2$, then $T_{\xi\otimes L^r}\in \ST_d^+$, so it is enough to take the union in equation \eqref{eq:unionOrbits} over $T\in \ST_d^+$

By the previous Lemma, $\Phi(V\cdot (E,E_\bullet))$ is covered by a disjoint finite union of orbits of $J(X)\cdot T(E,E_\bullet)$, which are irreducible constructible sets. Thus, there must exist some $T\in \ST_d$ such that $\Phi(V\cdot (E,E_\bullet))\cap J(X)\cdot T(E,E_\bullet)$ is open and dense in $J(X)\cdot T(E,E_\bullet)$. Let $U_T\subset J(X)$ be the open subset such that $J_T: U_T \longrightarrow U_T\cdot T(E,E_\bullet)$ is an isomorphism. Let
$$\tilde{\SP}=J_{\id}^{-1}(\Phi^{-1}(U_T\cdot (E,E_\bullet) \cap \Phi(V\cdot (E,E_\bullet))))$$
$$\tilde{\SP'}=J_T^{-1}(\Phi(\tilde{\SP}\cdot (E,E_\bullet))) \subset J_T^{-1}(U_T\cdot (E,E_\bullet))=U_T$$
Then, by construction, $J_{\id}:\tilde{\SP} \longrightarrow \tilde{\SP}\cdot (E,E_\bullet)$ and $J_T:\tilde{\SP}' \longrightarrow \tilde{\SP}' \cdot T(E,E_\bullet)$ are isomorphisms and $\Phi(\tilde{\SP}\cdot (E,E_\bullet))=\tilde{\SP}'\cdot T(E,E_\bullet)$.
\end{proof}

The previous Lemma bounds the possible images of the orbits of the Jacobian, and provides the rigidity result needed to obtain the following classification result and prove the main Theorem.

\begin{lemma}
\label{lemma:recoverAction}
Let $(X,D)$ be a smooth complex projective curve of genus $g\ge 6$. Let $\Phi:\SM(r,\alpha,d)\dashrightarrow \SM(r,\alpha,d)$ be a $3$-birational map. Then there exist
\begin{itemize}
\item an open subset $\SU\subset \SM(r,\alpha,d)$ where $\Phi$ is defined,
\item a basic transformation $T\in \ST_d$, which, moreover, can be chosen in $\ST_d^+$ if $r=2$,
\item a line bundle $\xi\in \Pic^d(X)$, and
\item an automorphism $\rho\in \Aut(J(X),J(X)[r])$,
\end{itemize}
such that for each $(E,E_\bullet)\in \SU$
$$\Phi(E,E_\bullet)\cong T(\SA_\rho^\xi(E,E_\bullet))$$
\end{lemma}

\begin{proof}
By Lemma \ref{lemma:JacobianOrbitsFixed} and Lemma \ref{lemma:basicFaithful}, there exists $\xi\in \Pic^d(X)$ and a generic $(E,E_\bullet)\in \SM(r,\alpha,\xi)$ such that $(E,E_\bullet)$ is $\alpha'$-stable for any generic system of weights $\alpha'$, it is not fixed by nontrivial $T\in \ST$ if $r>2$ or $T\in \ST^+$ if $r=2$, and there exists open subsets $\tilde{\SP},\tilde{\SP}'\subset J(X)$ such that the composition
$$\tilde{\SP} \stackrel{J_{\id}}{\longrightarrow} \tilde{\SP}\cdot (E,E_\bullet) \stackrel{\Phi}{\longrightarrow} \tilde{\SP}' \cdot T(E,E_\bullet) \stackrel{J_T^{-1}}{\longrightarrow} \tilde{\SP}'$$
is an isomorphism. Composing with the $3$-birational map $T^{-1}:\SM(r,\alpha,d)\dashrightarrow \SM(r,\alpha,d)$ (which is well defined over $\SP'\cdot T(E,E_\bullet)$ by the choice of $(E,E_\bullet)$) and repeating the construction of the previous Lemma, we obtain that there is an open subset $\tilde{\SP}''\subset J(X)$ such that the composition
$$\tilde{\SP} \stackrel{J_{\id}}{\longrightarrow} \tilde{\SP}\cdot (E,E_\bullet) \stackrel{T^{-1}\circ \Phi}{\longrightarrow} \tilde{\SP}'' \cdot (E,E_\bullet) \stackrel{J_{\id}^{-1}}{\longrightarrow} \tilde{\SP}''$$
is an isomorphism, so it induces a birational map $J(X)\dashrightarrow J(X)$. As $J(X)$ is abelian, this map extends to a unique automorphism $\rho:J(X) \longrightarrow J(X)$ (c.f. \cite[Theorem 3.8]{Milne08}). Then for each $L\in \tilde{\SP}$ we have
$$(T^{-1}\circ\Phi)((E,E_\bullet) \otimes L) \cong (E,E_\bullet) \otimes \rho(L)$$

The map $\rho$ is the composition of a translation and a group homomorphism (c.f. \cite[Corollary 1.2]{Milne08}). Changing $T$ to $T\circ \ST_M$ for a suitable $M\in J(X)$, it is clear that we can choose $T$ such that $\rho(\SO_X)=\SO_X$. Thus, we can assume that $\rho$ is a group homomorphism. Let us verify that $\rho$ fixes the $r$-torsion of the Jacobian, $J(X)[r]$. Let us consider the following open subset of $J(X)$.
$$\SW=\bigcap_{S\in J(X)[r]} S\cdot \tilde{\SP} \subset \tilde{\SP}$$
As $J(X)[r]$ is finite, then $\SW$ is an open nonempty subset such that for each $L\in \SW$ and $S\in J(X)[r]$, $L\otimes S\in \SW$. Fix $L\in \SW$. Then for each $S\in J(X)[r]$,
$$(T^{-1}\circ\Phi)((E,E_\bullet)\otimes L \otimes S) \cong (E,E_\bullet)\otimes \rho(L\otimes S) = (E,E_\bullet)\otimes \rho(L)\otimes \rho(S)$$
On the other hand, by Lemma \ref{lemma:basicTransFibres} we know that for each $\xi'\in \det(\SP\cdot (E,E_\bullet))$ there exists a basic transformation $T_{\xi'}\in \ST_d$ (which can be taken in $\ST^+_d$ if $r=2$), such that for each $(E',E'_\bullet)\in \SM(r,\alpha,\xi')$ where $T^{-1}\circ\Phi$ is defined, $(T^{-1}\circ\Phi)(E',E'_\bullet)\cong T_{\xi'}(E',E'_\bullet)$. Thus, for each $S\in J(X)[r]$,
$$T_{\xi\otimes L^r}((E,E_\bullet)\otimes L \otimes S)\cong (T^{-1}\circ\Phi)((E,E_\bullet)\otimes L \otimes S) \cong (E,E_\bullet)\otimes \rho(L)\otimes \rho(S)$$
For $S=\SO_X$, we have that
$$T_{\xi\otimes L^r}((E,E_\bullet)\otimes L)\cong (E,E_\bullet)\otimes \rho(L) = (E,E_\bullet)\otimes L \otimes (\rho-\id)(L)$$
Since we took $(E,E_\bullet)$ such that it is not fixed by any nontrivial basic transformation, we conclude that for each $L\in \SW$, $T_{\xi\otimes L^r}=\ST_{(\rho-\id)(L)}$ (for $r=2$, observe that we can choose $T_{\xi\otimes L^r}\in \ST^+$ and $\ST_{(\rho-\id)(L)}\in \ST^+$). Thus, for any other $S\in J(X)[r]$,
$$(T^{-1}\circ\Phi)((E,E_\bullet)\otimes L \otimes S) \cong  (E,E_\bullet)\otimes \rho(L)\otimes \rho(S) = T_{\xi\otimes L^r}((E,E_\bullet)\otimes L \otimes S)= (E,E_\bullet) \otimes \rho(L) \otimes S$$
which implies that that $\rho(S)=S$ for each $S\in J(X)[r]$, so $\rho\in \Aut(J(X),J(X)[r])$. Following the argument from Section \ref{section:action}, we can write $\rho=\id+r\wt{\rho}$, and for each $L\in \SW$
\begin{multline*}
(T^{-1}\circ\Phi)((E,E_\bullet)\otimes L) \cong (E,E_\bullet) \otimes \rho(L) = (E,E_\bullet) \otimes L \otimes (\rho-\id)(L) = (E,E_\bullet) \otimes L\otimes \wt{\rho}(L^r)\\
=(E,E_\bullet) \otimes L \otimes \wt{\rho}(\det(E)\otimes L) \otimes \xi^{-1})= \SA_\rho^\xi((E,E_\bullet)\otimes L)
\end{multline*}
In particular, for each $L\in W$, we have
$$T_{\xi\otimes L^r}=\ST_{(\rho-\id)(L)} = \ST_{\wt{\rho}(L^r)} = \ST_{\wt{\rho}(\xi\otimes L^r \otimes \xi^{-1})}$$
This implies that for each $\xi'\in \det(\SW\cdot (E,E_\bullet))$ we have
$$T_{\xi'}=\ST_{\wt{\rho}(\xi'\otimes \xi^{-1})}$$
which implies that for each $(E',E'_\bullet)\in \SM(r,\alpha,\xi')$ for which $T^{-1}\circ\Phi$ is defined
$$(T^{-1}\circ\Phi)(E',E'_\bullet) \cong T_{\xi'}(E',E'_\bullet) = (E',E'_\bullet) \otimes \wt{\rho}(\xi'\otimes \xi^{-1}) = (E',E'_\bullet)\otimes \wt{\rho}(\det(E') \otimes \xi^{-1}) = \SA_{\rho}^\xi(E',E'_\bullet)$$
Taking $\SU\subset \SM(r,\alpha,d)$ as intersection of $\det^{-1}(\det(\SW\cdot (E,E_\bullet)))$ and the open subset where $T^{-1}\circ\Phi$ is defined, we obtain the desired result.
\end{proof}

\begin{theorem}
\label{thm:3birMaps}
Let $(X,D)$ and $(X',D')$ be two smooth complex projective curves of genus $g\ge 6$ and $g'\ge 6$ respectively with sets of marked points $D\subset X$ and $D'\subset X'$. Let $\alpha$ and $\alpha'$ be full flag generic systems of weights over $(X,D)$ and $(X',D')$ respectively. Let $\Phi: \SM(X,D,r,\alpha,d)\dashrightarrow \SM(X',D',r',\alpha',d')$ be a $3$-birational map. Then $r=r'$ and there exist
\begin{itemize}
\item an isomorphism $\sigma:(X,D)\longrightarrow (X',D')$,
\item a basic transformation $T\in \ST$, which belongs to $\ST^+$ if $r=2$,
\item a line bundle $\xi\in \Pic^d(X)$ and
\item an automorphism $\rho\in \Aut(J(X),J(X)[r])$
\end{itemize}
such that $T(d)=d'$ and for each $(E,E_\bullet)\in \SM(X,D,r,\alpha,d)$ where $\Phi$ is defined
$$\sigma^*\Phi(E,E_\bullet) \cong (T\circ \SA^\xi_\rho)(E,E_\bullet)$$
\end{theorem}

\begin{proof}
By Torelli theorem \ref{thm:TorelliBir}, we know that $r=r'$ and there exists an isomorphism $\sigma:(X,D)\longrightarrow (X',D')$. Composing $\Phi$ with the pullback by $\sigma$, and substituting $\alpha'$ by $\sigma^*\alpha'$, we can assume without loss of generality that $(X,D)=(X',D')$. Moreover, by Lemma \ref{lemma:reciprocalTorelli}, there exists a basic transformation $T\in \ST^+$ inducing a $3$-birational map $\SM(X,D,r,\alpha,d)\dashrightarrow \SM(X,D,r,\alpha',d')$. Composing with $T^{-1}$, we can assume that $\alpha=\alpha'$ and $d=d'$.

Let $\SU,\SU'\subset \SM(X,D,r,\alpha,d)=\SM(r,\alpha,d)$ be the open subsets such that $\Phi:\SU \longrightarrow \SU'$ is an isomorphism. We can apply Lemma \ref{lemma:recoverAction} to find an open subset $\SU_1\subset \SU\subset \SM(r,\alpha,d)$, a basic transformation $T\in \ST_d$ (which belongs to $\ST_d^+$ if $r=2$), a line bundle $\xi\in \Pic^d(X)$ and $\rho\in \Aut(J(X),J(X)[r])$ such that for all $(E,E_\bullet)\in \SU_1$
$$\Phi(E,E_\bullet)\cong (T\circ \SA^\xi_\rho)(E,E_\bullet)$$
We know that $\Phi|_{\SU_1}=(T\circ\SA^\xi_\rho)|_{\SU_1}$. It is only left to prove that these two maps agree on $\SU$. Let $\SV\subset \SM(r,\alpha,d)$ be the subset parameterizing parabolic vector bundles which are both $\alpha$-stable and $T^{-1}(\alpha)$-stable. By \cite[Lemma 2.3]{AG19}, $\SV$ is an open dense subset whose complement has codimension at least $3$ in $\SM(r,\alpha,d)$. Observe that $\SA^\xi_\rho$ does not change the stability of the parabolic vector bundle. Thus, by the choice of $\SV$, $T\circ \SA^\xi_\rho$ extends to a map $T\circ \SA^\xi_\rho: \SV\longrightarrow \SM(r,\alpha,d)$. Then, $T\circ \SA^\xi_\rho|_{\SV\cap \SU}$ and $\Phi|_{\SV\cap \SU}$ are two possible maps $\SV\cap \SU \longrightarrow \SM(r,\alpha,d)$ whose restriction to the dense subset $\SU_1\subset \SV\cap \SU$ coincide. As $\SM(r,\alpha,d)$ is separable, we conclude that $T\circ \SA^\xi_\rho|_{\SV\cap \SU}=\Phi|_{\SV\cap \SU}$.

On the other hand, by \cite[Proposition 3.2]{BY99}, $\SM(r,\alpha,d)$ is a fine moduli space. Let $(\SE,\SE_\bullet)\longrightarrow \SM(r,\alpha,d) \times X$ be a universal family. Then the map $\Phi:\SU\longrightarrow \SM(r,\alpha,d)$ must be represented by a family $(\SE',\SE_\bullet')=\Phi^*(\SE,\SE_\bullet) \longrightarrow \SU\times X$ of $\alpha$-stable quasi-parabolic vector bundles. Tensoring with a line bundle on $\SV\cap \SU$ if necessary, we can assume that
$$(\SE',\SE_\bullet')|_{(\SV\cap \SU) \times X} \cong (T\circ \SA^\xi_\rho)(\SE,\SE_\bullet)|_{(\SV\cap \SU)\times X}$$
As $T\circ \SA^\xi_\rho$ is well defined on arbitrary families of quasiparabolic vector bundles, $(T\circ \SA^\xi_\rho)(\SE,\SE_\bullet)|_{\SU \times X}$ is a well defined family of quasi-parabolic vector bundles over $\SU$. Thus, $(\SE',\SE'_\bullet)$ and $(T\circ \SA^\xi_\rho)(\SE,\SE_\bullet)$ are two different families of quasi-parabolic bundles on $\SU\times X$ which are isomorphic over $(\SU\cap \SV)\times X$.

However, $\SU$ is a smooth complex scheme and $\SU\cap \SV$ is an open subset whose complement has codimension at least $3$, so by \cite[Lemma 2.8]{AG19}, there is at most one extension of $(\SE',\SE'_\bullet)|_{(\SU\cap \SV)\times X}$ to $\SV\times X$, and we conclude that $(\SE',\SE'_\bullet) \cong (T\circ \SA^\xi_\rho)(\SE,\SE_\bullet)$. Thus, $\Phi(E,E_\bullet)\cong T(\SA_\rho^\xi(E,E_\bullet))$ for all $(E,E_\bullet)\in \SU$.
\end{proof}

\begin{corollary}
\label{cor:maps}
Let $(X,D)$ and $(X',D')$ be two smooth complex projective curves of genus $g\ge 6$ and $g'\ge 6$ respectively with sets of marked points $D\subset X$ and $D'\subset X'$. Let $\alpha$ and $\alpha'$ be full flag generic systems of weights over $(X,D)$ and $(X',D')$ respectively. Let $\Phi: \SM(X,D,r,\alpha,d)\longrightarrow \SM(X',D',r',\alpha',d')$ be an isomorphism. Then $r=r'$ and there exist
\begin{itemize}
\item an isomorphism $\sigma:(X,D)\longrightarrow (X',D')$,
\item a basic transformation $T\in \ST$, which belongs to $\ST^+$ if $r=2$,
\item a line bundle $\xi\in \Pic^d(X)$ and
\item an automorphism $\rho\in \Aut(J(X),J(X)[r])$
\end{itemize}
such that $T(d)=d'$, $T(\alpha)$ and $\sigma^*\alpha'$ belong to the same stability chamber and for each $(E,E_\bullet)\in \SM(X,D,r,\alpha,d)$
$$\sigma^*\Phi(E,E_\bullet) \cong (T\circ \SA^\xi_\rho)(E,E_\bullet)$$
\end{corollary}

\begin{proof}
As $\Phi$ is an isomorphism defined over the whole moduli space, Theorem \ref{thm:3birMaps} implies that $r=r'$ and there exists $\xi\in \Pic^d(X)$, $T\in \ST$ (which can be chosen in $\ST^+$ if $r=2$) such that $T(d)=d'$ and $\rho\in \Aut(J(X),J(X)[r])$ such that for each $(E,E_\bullet)\in \SM(X,D,r,\alpha,d)$
$$\sigma^*\Phi(E,E_\bullet) \cong (T\circ \SA^\xi_\rho)(E,E_\bullet)$$
As $\Phi$ is an isomorphism, when $(E,E_\bullet)$ varies in $\SM(X,D,r,\alpha,d)$, $\Phi(E,E_\bullet)$ takes all possible values in $\SM(X',D',r,\alpha',d')$, so $\sigma^*\Phi(E,E_\bullet)$ ranges over all possible isomorphism classes of $\sigma^*\alpha'$-stable parabolic vector bundles. On the other hand, as $\SA^\xi_\rho$ does not change the stability condition, $(T\circ \SA^\xi_\rho)(E,E_\bullet)$ is $T(\alpha)$-stable for each $\alpha$-stable $(E,E_\bullet)$. As $(E,E_\bullet)$ changes over $\SM(X,D,r,\alpha,d)$, $(T\circ \SA^\xi_\rho)(E,E_\bullet)$ covers all possible $T(\alpha)$-stable bundles. Thus, the sets of isomorphism classes of $\sigma^*\alpha'$-stable and $T(\alpha)$-stable parabolic vector bundles coincide, so $\sigma^*\alpha'$ and $T(\alpha)$ belong to the same stability chamber.
\end{proof}
Finally, let us analyze the group structure of the automorphisms. First of all, observe that for each $L\in J(X)$, $\xi\in \Pic^d(X)$ and $\rho\in \Aut(J(X),J(X)[r])$
\begin{multline*}
\SA_\rho^\xi((E,E_\bullet)\otimes L)=(E,E_\bullet)\otimes L \otimes \wt{\rho}(\det(E)\otimes L^r \otimes \xi^{-1})\\
=(E,E_\bullet)\otimes \wt{\rho}(\det(E)\otimes \xi^{-1})\otimes L \otimes \wt{\rho}(L^r)=\SA_\rho^\xi(E,E_\bullet)\otimes \rho(L),
\end{multline*}
so
$$\SA_\rho^\xi \circ \ST_L=\ST_{\rho(L)}\circ \SA_\rho^\xi.$$
Now, let $\xi\in \Pic^d(X)$, $\rho\in \Aut(J(X),J(X)[r])$ and $T=(\sigma,s,L,H)\in \ST_d$. Let $\rho_\sigma=\sigma^*\circ \rho \circ (\sigma^{-1})^*$. Then observe that for each $(E,E_\bullet)\in \SM(r,\alpha,\xi')\subset \SM(r,\alpha,d)$ we have
\begin{multline*}
T(\SA_\rho^\xi(E,E_\bullet))=T((E,E_\bullet)\otimes \wt{\rho}(\det(E)\otimes \xi^{-1}))= T(E,E_\bullet) \otimes \sigma^*(\wt{\rho}(\xi'\otimes \xi^{-1}))^s\\
=T(E,E_\bullet) \otimes \sigma^*(\wt{\rho}((\xi')^s \otimes \xi^{-s})) = T(E,E_\bullet) \otimes \wt{\rho_\sigma}(\sigma^*(\xi')^s \otimes \sigma^*\xi^{-s})\\
=T(E,E_\bullet) \otimes \wt{\rho_\sigma}(\sigma^*(\xi'\otimes L^r(-H))^s \otimes \sigma^*(\xi\otimes L^r(-H))^{-s})\\
=T(E,E_\bullet)\otimes \wt{\rho_\sigma}(T(\xi') \otimes T(\xi)^{-1})\\
= T(E,E_\bullet)\otimes \wt{\rho_\sigma}(\det(T(E,E_\bullet)) \otimes \xi^{-1}) \otimes \wt{\rho_\sigma}(\xi\otimes T(\xi)^{-1})\\
=\ST_{\wt{\rho_\sigma}(\xi\otimes T(\xi)^{-1})} \circ \SA_{\rho_\sigma}^\xi \circ T(E,E_\bullet) =\SA_{\rho_\sigma}^\xi \circ \ST_{\rho_\sigma^{-1}(\rho(\xi\otimes T(\xi^{-1})))}\circ T (E,E_\bullet).
\end{multline*}

This computation motivates the following definition.

\begin{definition}
An extended basic transformation of degree $d$ is a tuple $T=(\rho,\sigma,s,L,H)$ where $(\sigma,s,L,H)\in \ST_d$ is a basic transformation and $\rho\in \Aut(J(X),J(X)[r])$.
Fix once and for all a line bundle $\xi\in \Pic^d$. Let $\overline{\ST}_d$ be the group of extended basic transformations of degree $d$, generated by
\begin{itemize}
\item $\Sigma_\sigma=(\id,\sigma,1,\SO_X,0)$
\item $\SD^+=(\id,\id,1,\SO_X,0)=\id_\ST$
\item $\SD^-=(\id,\id,-1,\SO_X,0)$
\item $\ST_L=(\id,\id,1,L,0)$
\item $\SH_H=(\id,\id,1,\SO_X,H)$
\item $\SA_\rho^\xi=(\rho,\id,1,\SO_X,0)$
\end{itemize}
with the same relations as Lemma \ref{lemma:compositionRules} plus the following two relations.
For each $\rho,\rho'\in \Aut(J(X),J(X)[r])$
$$\SA_\rho^\xi \circ \SA_{\rho'}^\xi=\SA_{\rho\circ \rho'}^\xi$$
and for each $T=(\id,\sigma,s,L,H) \in \ST_d\subset \overline{\ST}_d$ and each $\rho\in \Aut(J(X),J(X)[r])$
$$T\circ \SA_\rho^\xi = \SA_{\rho_\sigma}^\xi \circ \ST_{\rho_\sigma^{-1}(\rho(\xi\otimes T(\xi^{-1})))}\circ T$$
Let $\overline{\ST}_d^+\subset \overline{\ST}_d$ be the subset of transformations where $s=1$ and $\overline{\ST}_{d,\alpha}$ the subset of transformations such that $(\sigma,s,L,H)\in \ST_{d,\alpha}$. Finally, let $\overline{\ST}_{d,\alpha}^+=\overline{\ST}_d^+\cap \overline{\ST}_{d,\alpha}$.
\end{definition} 

As all elements of $\ST_d$ give $3$-birational automorphisms of $\SM(r,\alpha,d)$ and $\SA_\rho^\xi$ acts by automorphisms of the moduli space by Lemma \ref{lemma:autoJacobianAction}, it is clear that all transformations $T\in \overline{\ST}_d$ induce $3$-birational automorphisms of $\SM(r,\alpha,d)$, which extend to automorphisms whenever $T\in \overline{\ST}_{d,\alpha}$ and we have the following theorem.
\begin{theorem}
\label{thm:automorphisms}
Let $(X,D)$ be a smooth complex projective curve of genus $g\ge 6$ and let $\alpha$ be a full flag system of weights over $(X,D)$. Then if $r>2$
$$\Aut_{3-\op{bir}}(\SM(r,\alpha,d))= \overline{\ST}_d, \quad \quad \Aut(\SM(r,\alpha,d))= \overline{\ST}_{d,\alpha}$$
and if $r=2$ then
$$\Aut_{3-\op{bir}}(\SM(2,\alpha,d))= \overline{\ST}_d^+,\quad \quad \Aut(\SM(2,\alpha,d))= \overline{\ST}_{d,\alpha}^+$$
\end{theorem}

\begin{proof}
Theorem \ref{thm:3birMaps} proves that for each $3$-birational automorphism $\Phi$ there exits $\xi'\in \Pic^d(X)$, $\rho\in \Aut(J(X),J(X)[r])$ and $T\in \ST_d$ (which can be taken in $\ST_d^+$ if $r=2$) such that $\Phi=\SA_\rho^{\xi'}\circ T$. By Remark \ref{rmk:changexi}, we can choose $T$ so that $\xi'=\xi$ is fixed for any automorphism. This, together with the previous computations and the definition of the group $\overline{\ST}_d$, clearly shows that $\overline{\ST}_d$ (respectively $\overline{\ST}_d^+$ if $r=2$) surjects as a group onto $\Aut_{3-\op{bir}}(\SM(r,\alpha,d))$. Similarly, Corollary \ref{cor:maps} proves that $\overline{\ST}_{d,\alpha}$ ($\overline{\ST}_{d,\alpha}^+$ if $r=2$) surjects onto $\Aut(\SM(r,\alpha,d))$.

Therefore, it is enough to show that there do not exist two different $T,T'\in \overline{\ST}_d$ (respectively, in $\overline{\ST}_d^+$ if $r=2$) which induce the same $3$-birational automorphism of $\SM(r,\alpha,d)$. Observe that this also implies that two different $T,T'\in \overline{\ST}_{d,\alpha}$ (or in $\overline{\ST}_{d,\alpha}^+$ if $r=2$) cannot induce the same automorphism of the moduli space. Composing with $(T')^{-1}$, this is equivalent to proving that there does not exist a nontrivial $T\in \overline{\ST}_d$ (respectively, $\overline{\ST}_d^+$) which acts as the identity on $\SM(r,\alpha,d)$.

Assume the contrary. Let $\SA_\rho^\xi \circ T$ be a transformation with $T\in \ST_d$ such that for each $(E,E_\bullet)\in \SM(r,\alpha,d)$ where $T$ is defined, $\SA_\rho^\xi\circ T(E,E_\bullet)\cong (E,E_\bullet)$ and such that $T\in \ST_d^+$ if $r=2$.  As $\SA_\rho^\xi$ fixes $\SM(r,\alpha,\xi)$, then $T(E,E_\bullet)\cong (E,E_\bullet)$ for each $(E,E_\bullet)\in \SM(r,\alpha,\xi)$ where $T$ is defined. By Lemma \ref{lemma:basicFaithful}, this implies $T=\id$. Thus, $\SA_\rho^\xi$ must act as the identity on $\SM(r,\alpha,d)$. A direct computation taking determinants (see the proof of Lemma \ref{lemma:autoJacobianAction}) shows that the induced map by $\SA_\rho^\xi$ on $\Pic^d(X)$ is given by
$$\SA_\rho^\xi(\xi')=\rho(\xi'\otimes \xi^{-1})\otimes \xi$$
for each $\xi'\in \Pic^d(X)$. If $\SA_\rho^\xi$ is the identity, this map must be the identity too, but then 
$$\xi'\cong \SA_\rho^\xi(\xi') =\rho(\xi'\otimes \xi^{-1})\otimes \xi$$
so for all $\xi'\in \Pic^d(X)$
$$\xi'\otimes \xi^{-1}\cong \rho(\xi'\otimes \xi^{-1})$$
As $\xi'\otimes \xi^{-1}$ runs over all elements of the Jacobian $J(X)$, this implies $\rho=\id$.
\end{proof}

\begin{remark}
For $r=2$, observe that the duality map $\op{inv}(L)=L^{-1}\in \Aut(J(X))$ fixes the $2$-torsion $J(X)[2]$, so $\op{inv}\in \Aut(J(X),J(X)[r])$. We have $\wt{\op{inv}}=\op{inv}$ and, thus, for each $(E,E_\bullet)$
$$\SA_{\op{inv}}^\xi (E,E_\bullet)= (E,E_\bullet) \otimes \det(E)^{-1} \otimes \xi$$
As $(E,E_\bullet)\otimes \det(E)^{-1} \cong (E,E_\bullet)^\vee$ for all quasiparabolic vector bundles, we obtain that for $r=2$
$$\SA_{\op{inv}}^\xi (E,E_\bullet)  \cong \ST_\xi \circ \SD^- (E,E_\bullet)$$
This explains the fact that we can always write $3$-birational automorphisms of $\SM(2,\alpha,d)$ as extended basic transformations which do not involve the duality operator $\SD^-$.
\end{remark}

\printbibliography

\end{document}